\documentclass[12pt]{article}
\usepackage{amsmath}
\usepackage{amsthm}
\usepackage{paralist}
\usepackage{makeidx}
\usepackage{amssymb}
\title{Brouwer's fan theorem}
\author{Josef Berger}
\makeindex
\def\X{\{0,1\}^\N}
\def\Xs{\{0,1\}^*}
\def\N{\mathbb{N}}
\newcommand{\mr}[1]{\mathrm{#1}}
\newcommand{\av}[1]{\left|{#1}\right|}
\newcommand{\brac}[1]{\left\{{#1}\right\}}
\newcommand{\para}[1]{\left({#1}\right)}
\newcommand{\f}[2]{{\overline{#1}{#2}}}

\newcommand{\deef}{\, : \Leftrightarrow \;}
\def\zo{\{0,1\}}
\def\a{\alpha}
\def\g{\gamma}
\def\b{\beta}
\newcommand{\ca}[1]{{\cal #1}}
\def\cfan{\mr{c}\text{-}\mr{FAN}}
\def\uc{\mr{UC}}
\def\muc{\mr{MUC}}
\def\llpo{\mathrm{LLPO}}

\def\lpl{\mathrm{LPL}}
\theoremstyle{plain}
\def\wkl{\mathrm{WKL}}
\def\fan{\mathrm{FAN}}

\newtheorem{lemma}{Lemma}
\newtheorem{proposition}{Proposition}
\newtheorem{corollary}{Corollary}

\begin{document}
\maketitle
\thispagestyle{empty}

\begin{abstract}
\noindent
Brouwer's fan theorem states that every
bar is a uniform bar.
We give an overview of the status of this axiom in   
Bishop's constructive mathematics.
In particular, we describe the relationship between the fan theorem, the weak K\"onig lemma, and the 
uniform continuity theorem.
\end{abstract}


\bigskip


\section{Introduction}

We analyse
the relationship between Brouwer's fan theorem for detachable bars, the weak K\"onig lemma, and the uniform continuity
theorem:
\begin{description}
\item[$\fan$] Every detachable bar is a uniform bar.
\item[$\wkl$] Every infinite tree has a path.
\item[$\uc$] Every pointwise continuous function $F:\X \to \N$ is uniformly continuous.
\end{description}

\bigskip

Extending the well-known basic picture
\[
\wkl \Rightarrow \uc \Rightarrow \fan,
\]
we prove the implications of the following diagram:
\[
\begin{array}{ccccccc}
  \wkl & \Rightarrow &\fan \land\mr{DEFU} & \Leftrightarrow &\cfan&\Rightarrow&\fan \\
  &&&&&&\\
  
  & & \Updownarrow &&  \Updownarrow && \\
    &&&&&&\\
& &\fan \land \mr{DECO} & \Leftrightarrow&\uc &&
\end{array}
\]

The axiom $\cfan$ was introduced in \cite{the_logical_pale} as a version of the fan theorem
which is equivalent to $\uc$. The axioms $\mr{DEFU}$ and $\mr{DECO}$ are auxiliary statements clarifying
the gap between $\fan$ and $\cfan$, $\uc$ respectively. 

\bigskip

Mosts results mentioned in this article have been published before, see
\begin{itemize}
\item \cite{hawkl,ishihara_weak} for results about $\wkl$ and $\fan$,
\item \cite{berger_bridges_schuster_pale,berger_ishihara_pale} for results about unique existence,
and 
\item \cite{the_fan_pale,the_logical_pale,dbfr} for results about $\uc$.
\end{itemize}

\section{Notation}

Let $\N=\brac{0,1,2,\ldots}$ denote the set of natural numbers.
We use the variables $k,l,m,n,N,i,j$ for natural numbers.
Let $\Xs$ denote the set of all finite binary sequences, typically denoted by variables like  $u,v,w$.
The symbol $\o$ denotes the empty sequence. Note that $\o \in \Xs$.
Let $\X$ denote the set of all infinite binary sequences. We use the variables $\a,\b,\g$ for elements of $\X$.
The \textit{length} of $u$ is denoted by $\av{u}$. 
We define 
$$\ca{N} = \brac{u \mid \av{u} > 0  \land \forall i < \av{u} \para{u_i=0}}$$
and
$$\ca{E} = \brac{u \mid \av{u} > 0  \land \forall i < \av{u} \para{u_i=1}}.$$
Let $\zo^n = \brac{u \mid \av{u}=n}$ denote the set of all binary sequences of lenght $n$.
The \textit{concatenation} of $u$ and $v$ is denoted by $ u* v$ and defined by
\[
(u_0,\ldots,u_{n-1})* (v_0,\ldots,v_{m-1}) = (u_0,\ldots,u_{n-1},v_0,\ldots,v_{m-1}).
\]
The \textit{concatenation} of $u$ and $\a$ is denoted by $ u* \a$ and defined by
\[
(u_0,\ldots,u_{n-1})* \a  = (u_0,\ldots,u_{n-1},\a_0,\a_1,\a_2,\ldots).
\]  
If $n \le \av{w}$, we denote the \textit{restriction} of $w$ to the first $n$ elements by $\f{w}{n}$. Note that
  $\f{w}{0}  = \o$.
We denote the \textit{restriction} of $\a$ to the first $n$ elements by $\f{\a}{n}$.
We write $u < v$ for $$\av{u} = \av{v}  \land 
\exists k < \av{u} \para{\f{u}{k}=\f{v}{k}  \land  u_k=0  \land v_k=1 }.$$

Fix a subset $A$ of $\Xs$. We define
\begin{itemize}
\item $A^\circ=  \brac{u  \mid \forall w \para{u*w \in A}}$,
\item $\overline{A} =  \brac{u * w \mid u \in A \land w \in \Xs}$, and
\item for any $u$, $A_u =\brac{w \mid u*w \in A}$.
\end{itemize}

A \textit{path}\index{path of a tree} of $A$ is a sequence $\a$ such that
$\forall n \para{\f{\a}{n} \in A}$. The set $A$ hat \textit{at most one path}\index{at most one path} if
\[
\forall \a,\b\para{ \exists n \para{\a_n  \ne \b_n}  \Rightarrow 
\exists n
\para{
\f{\a}{n} \notin A 
  \lor  \f{\b}{n} \notin A  } }.
\]
A \textit{longest path}\index{longest path}  of $A$ is a sequence $\a$ such that
$\forall
u  \in A \para{\f{\a}{\av{u}} \in A}$. 
Note that each path of $A$ is a longest path of $A$ but not necessarily vice versa.

\bigskip

The set $A$ is said to be
\begin{itemize}
\item \textit{detachable}\index{detachable set} (from $\Xs$) if $$\forall u \para{u \in A   \lor  u \notin A};$$
\item a \textit{c-set}\index{c-set} if there exists a detachable set $D \subseteq \Xs$ such that
$A=D^\circ$;
\item \textit{closed under extension} if
\[
\forall u  \in A \, \forall w \para{u*w \in A};
\]
\item \textit{closed under restriction} if
\[
\forall u,w  \para{u*w \in A \Rightarrow  u \in A};
\]
\item a \textit{tree}\index{tree} if it is detachable and closed under restriction;
\item \textit{infinite}\index{infinite tree} if
  $$\forall n \, \exists u \para{\av{u}=n \land u \in A};$$
\item a \textit{bar}\index{bar} if $$\forall \a \, \exists n \para{\f{\a}{n} \in A};$$
\item a \textit{uniform bar}\index{uniform bar} if
$$\exists N \, \forall \a \, \exists n \le N \para{\f{\a}{n} \in A};$$
\item \textit{convex}\index{convex set} if
\[
\forall u,w \in A \, \forall v \para{u < v < w  \Rightarrow  v \in A};
\]
\item \textit{co-convex}\index{co-convex set} if $\brac{u \mid u \notin A}$ is convex.
\end{itemize}

\section{The weak K\"onig lemma}

The \textit{weak K\"onig lemma}\index{weak K\"onig lemma} is the following axiom:
\begin{description}
\item[$\wkl$] Every infinite tree has a path.
\end{description}

\begin{lemma}\label{lpp} Let $T$ be a tree. The following are equivalent:
\begin{enumerate}[(a)]
\item $T$ is an infinite tree.
\item Every longest path of $T$ is a path of $T$.
\end{enumerate}  
\end{lemma}

The notion of longest path is commonly used in graph theory.
In the context of constructive mathematics, it was introduced in \cite{ishihara_weak},
along with the \textit{longest path lemma}\index{longest path lemma}:

\begin{description}
\item[$\lpl$] Every tree has a longest path.
\end{description}

When working with trees and longest paths, it is convenient to
have at hand an algorithm  which transforms  an arbitrary tree $T$ into an infinite tree $T'$,
where $T'$ is as closely related to $T$ as possible.
Given a tree $T$, we define subsets $L_T$ and $T'$ of $\Xs$ by 
\[
u \in L_T \Leftrightarrow
\para{u \in T \land \forall w \para{\para{\av{u} < \av{w} \lor u<w} \Rightarrow w \notin T}
\lor \para{\o \notin T \land u = \o}
}
\]
and 
\[
u \in T'  \Leftrightarrow 
u \in T \cup \brac{\o}  \lor  \exists v \in L_T \, \exists w \in \ca{N} \para{u=v*w}. 
\]

Note that
\begin{itemize}
\item if $T = \emptyset$ then $L_T=\brac{\o}$ and $T'=\cal{N}\cup \brac{\o}$;
\item if $T$ is infinite, then $L_T=\emptyset$ and $T'=T$;
\item if $\o \in T$ and the set
\[
\brac{\av{u} \mid u \in T}
\]
  is bounded, then
   $L_T=\brac{u}$, where $u$ is an element of $T$ of maximal length. 
\end{itemize}  

\begin{proposition}\label{uxuxux} For every tree $T$ there is an infinite tree $T'$ such that
\begin{enumerate}[(a)]
\item $T \subseteq T'$;    
\item if $T$ is infinite, then $T=T'$;
\item if $T$ is convex, then $T'$ is convex as well;
\item if $T$ is finite, then $T'$ has at most one path;
\item every path of $T'$ is a longest path of $T$;
\item\label{lppd} if $T$ has at most one path, then $T'$ has at most one path.
\end{enumerate}
\end{proposition}

The \textit{lesser limited principle of omniscience}\index{lesser limited principle of omniscience} is the following axiom:

\begin{description}
\item[$\llpo$] For all $\a$,
  $$\forall n,m \para{\a_n = \a_m = 1 \Rightarrow n=m}  \Rightarrow \forall n \para{\a_{2n} = 0}
  \lor  \forall n \para{\a_{2n+1} = 0}.$$
\end{description}

\begin{proposition}\label{llpotowkl} $\llpo \Rightarrow \wkl$
\end{proposition}

\begin{proof} Let $T$ be an infinite tree. Let $\zeta_T(u,m)$ be defined by
\[
\zeta_T(u,m) \Leftrightarrow \exists w  \in \brac{0,1}^m \para{u * w \in T}.
\]
Moreover, set
\[
Z_T= \brac{u \mid \forall m \, \zeta_T(u,m)}.
\]

Define $\b$ by
\[
\b_{2n} = 0  \Leftrightarrow \zeta_T((0),n)
\]
and
\[
\b_{2n+1} = 0  \Leftrightarrow \zeta_T((1),n)
\]
and define $\a$ by
\[
\a_{n} = 1  \Leftrightarrow \b_{n} = 1 \land \forall m < n \para{\b_{m} = 0}.
\]
By $\llpo$, either
\begin{equation}\label{ichi}
\forall n \para{ \a_{2n} = 0 }
\end{equation}
or else
\begin{equation}\label{ni}
\forall n \para{ \a_{2n+1} = 0 }. 
\end{equation}

If \eqref{ichi} holds, we can conclude that
\[
\forall n \para{ \b_{2n} = 0 },
\]
which implies that $(0) \in Z_T$.

\bigskip

Analogously, we can show that \eqref{ni} implies that $(1) \in Z_T$.

\bigskip

Iterating this procedure yields a path $\g$ of $Z_T$.
Since $Z_T \subseteq T$, the sequence $\g$ is also  a path of $T$.

\end{proof}  
  
\begin{proposition}\label{notre} The following are equivalent:
\begin{enumerate}[(a)]
\item $\wkl$
\item $\lpl$
\item Every convex tree has a longest path.
\item Every infinite convex tree has a path.
\item $\llpo$
\end{enumerate}
\end{proposition}

\begin{proof} ``$(a)\Rightarrow (b)$'' This follows from Proposition \ref{uxuxux}.
  
\bigskip

``$(c)\Rightarrow (d)$'' Let $T$ be an infinite convex tree.
By $(c)$, the tree $T$ has a longest path $\a$.
By Lemma \ref{lpp}, the sequence $\a$ is a path of $T$.

\bigskip

``$(d)\Rightarrow (e)$'' Fix a sequence $\a$ with
\[
\forall n,m \para{ \a_n = \a_m = 1 \Rightarrow n=m}.
\]
We define an infinite convex tree $S$ by
\[
S  = \brac{\o,(0),(1)} \cup S_0  \cup S_1,
\]
where
\[
S_0 =  \brac{(0)*w \mid w \in \ca{E} \land \forall i \le \av{w} \para{\a_{2i}=0}},
\]
\[
S_1 =  \brac{(1)*w \mid w \in \ca{N} \land \forall i \le \av{w} \para{\a_{2i+1}=0}}.
\]
Let $\b$ be a path of $S$.
If $\b_1=0$, then $\forall i \para{\a_{2i}=0}$. If $\b_1=1$, then $\forall i \para{\a_{2i+1}=0}$.

\bigskip

``$(e)\Rightarrow (a)$''  This is Proposition \ref{llpotowkl}.

\end{proof}

Proposition \ref{llpotowkl} and
Proposition \ref{notre} (without the parts about convexity) were proved in \cite{hawkl} and \cite{ishihara_weak}.

\section{The fan theorem}

\textit{Brouwer's fan theorem}\index{Brouwer's fan theorem} is the following axiom:

\begin{description}
\item[$\fan$] Every detachable bar is a uniform bar.
\end{description}

\begin{lemma}\label{cava} If $B \subseteq \Xs$ is closed under extension, the following are equivalent:
\begin{enumerate}[(a)]  
\item $B$ is a uniform bar.
\item $\exists n \para{\zo^n \subseteq B}$
\end{enumerate}
\end{lemma}

The next lemma enables us to switch freely between a set $B$ itself and $\overline{B}$, when discussing bar-related properties.
This technique was introduced in \cite{ishihara_weak}.

\begin{lemma}\label{vier} Let $B$ be a subset $\Xs$. Then
\begin{enumerate}[(a)]
\item $\overline{B}$ is closed under extension,
\item $B \subseteq \overline{B}$,
\item $B$ is detachable  $\,\,\, \Rightarrow \,\,\,$ $\overline{B}$ is detachable,
\item $B$ is a bar $\,\,\, \Rightarrow \,\,\,$ $\overline{B}$ is a bar,
\item $\overline{B}$ is a uniform bar $\,\,\, \Rightarrow \,\,\,$ $B$
  is a uniform bar.
\end{enumerate}
\end{lemma}

\begin{proposition}\label{fangen} The following are equivalent:
\begin{enumerate}[(a)]
\item $\wkl!$
\item Each tree with at most one path has a longest path.
\item Each detachable bar which is closed under extension is a uniform bar.
\item $\fan$
\end{enumerate}
\end{proposition}

\begin{proof}  ``$(a) \Rightarrow (b)$'' This follows from Proposition \ref{uxuxux}.

\bigskip

``$(b) \Rightarrow (c)$'' Let $B$ be a detachable bar which is closed under extension.
Note that the set
\[
T = \brac{u \mid u \notin B}
\]
is a tree.
Since $B$ is a bar, the tree $T$ has at most one path.
Let $\a$ be a longest path of $T$.
There exists $n$ such that $\f{\a}{n} \in B$.
This implies that $T \cap \zo^n = \emptyset$ and therefore $\zo^n \subseteq B$.
Thus, by Lemma \ref{cava},
$B$ is a uniform bar.

\bigskip

``$(c) \Rightarrow (d)$'' This follows from Lemma \ref{vier}.

\bigskip

``$(d) \Rightarrow (a)$'' Fix 
an infinite tree $T$ with at most one path.
Define a set $B$ by
\[
u \in B\,  \deef  (0) * u \notin T   \lor  \forall
v  \in \zo^{\av{u}} \para{(1) * v \notin T}.
\]
We show that the $B$ is a bar.
Fix a sequence  $\a$.
We define
$B_\a \subseteq \Xs$ by
\[
v \in B_\a  \deef    \para{(0) * \f{\a}{\av{v}}
\notin T   \lor   (1) * v  \notin T}.
\]
Note that $B_\a$ is closed under extension.
Since $T$ has a most one path, $B_\a$
is a bar and therefore a uniform bar.
Thus, there exists $m$ such that
$\zo^m \subseteq B_\a$.
Thus, $\f{\a}{m} \in B$.
This concludes the proof that $B$ is a bar.

\bigskip

Since $B$ is is a uniform bar and closed under extension,
there exists an $n$ such that $\zo^n \subseteq B$. We show that either
\begin{enumerate}[(i)]
\item $\forall u \in  \zo^n \para{ (0) * u
   \notin T}$ or else
\item $\forall u \in  \zo^n \para{ (1) * u
  \notin T}$.
\end{enumerate}
If (i) fails to hold, there exists
$u \in \zo^n$ such that $(0)*u \in T$.
Since $u \in B$, this implies (ii).

\bigskip

Let $Z_T$ be defined as in the proof of Proposition \ref{llpotowkl}.
If (i) holds, then $(1) \in Z_T$. If (ii) holds, then $(0) \in Z_T$.
Iterating this construction yields a path $\gamma$ of $Z_T$ and therefore of $T$.

\end{proof}

The equivalence of $\fan$ and $\wkl!$ was proved in \cite{berger_ishihara_pale}.

\begin{corollary}\label{zitty}$\wkl \Rightarrow \fan$
\end{corollary}

Corollary \ref{zitty} was originally published in \cite{ishihara_weak}.

\begin{lemma}\label{lemacc} If $A \subseteq \Xs$ is co-convex, then
$\overline{A}$ is also co-convex.
\end{lemma}
\begin{proof} Fix $u, v, w$ such that $u < v < w$ and $u,w \notin \overline{A}$.
We have to show that $v \notin \overline{A}$.  
To this end, fix $m \le \av{v}$. Then $\f{u}{m}  \notin A$ and  $\f{w}{m}  \notin A$.
The co-convexity of $A$ yields $\f{v}{m} \notin A$.
Since $m \le \av{v}$ was chosen arbitrarily,
this implies that $v \notin \overline{A}$.

\end{proof}

\begin{proposition}\label{strang}$\,$ 
\begin{enumerate}[(a)]
\item Every infinite convex tree with at most one path has a path.
\item Every convex tree with at most one path has a longest path. 
\item\label{strangc} Every detachable co-convex bar is a uniform bar.
\end{enumerate}
\end{proposition}

\begin{proof} $(a)$ Let $T$ be an infinite convex tree with at most one path.
Set
\[
\a = \para{0,1,1,1,\ldots}
\]
and
\[
\b = \para{1,0,0,0,\ldots}.
\]
Since $T$ has at most one path, there exists $n$ such
that 
\[
\f{\a}{(n+1)} \notin T \lor 
\f{\b}{(n+1)} \notin T. 
\]
Since $T$ is convex, we can conclude that either
\begin{enumerate}[(i)]
\item $\forall u \in  \zo^n \para{ (0) * u \notin T}$ or else
\item $\forall u \in  \zo^n \para{ (1) * u \notin T}$.
\end{enumerate}  
In the first case, choose $\a_0=1$ and in the second case, choose $\a_0=0$.
Iterating this construction leads to a path $\a$ of $T$.

\bigskip

(b) follows from (a) and Proposition \ref{uxuxux}.

\bigskip

(c) Let $B$ be a detachable, co-convex bar.
By Lemma \ref{lemacc}, we can assume that $B$ is closed under extension.
The set
\[
T = \brac{u \mid u \notin B}
\]
is a convex tree.
Since $B$ is a bar, the tree $T$ has at most one path.
Let $\a$ be a longest path of $T$.
There exists $n$ such that $\f{\a}{n} \in B$.
Then $T \cap \zo^n = \emptyset$.
Thus, $\zo^n \subseteq B$.
This implies that
$B$ is a uniform bar.

\end{proof}

The notion of co-convex bars was introduced in \cite{three}.
In \cite{schwicht},
the fan theorem for co-convex bars---every detachable co-convex bar is a uniform bar---is used for 
a case study in program extraction from proofs.
For a treatment of convex trees in constructive reverse mathematics see \cite{miss}.

\section{The uniform continuity theorem}\label{poo}

Fix a function $F: \X \to \N$ and $u \in \Xs$. We define
\[
F(u) = F(u*{\textbf 0}),
\]
where $\textbf{0}=(0,0,0,\ldots)$. Moreover, we define a new function $F_u: \X \to \N$ by
\[
F_u(\a) = F(u*\a).
\]
For $n \in \N$, we define
\[
F \equiv n \Leftrightarrow  \forall \a\para{F(\a)=n}.
\]

A function  $F: \X \to \N$ is
\begin{itemize}
\item \textit{pointwise continuous}\index{pointwise continuous function} if
\[
\forall \a \, \exists n \, \forall \b 
\para{ \f{\a}{n} = \f{\b}{n} \,\,\, \Rightarrow \,\,\, F(\a) = F(\b)};
\]
\item \textit{uniformly continuous}\index{uniformly continuous function} if
\[
\exists n \, 
\forall \a,\b 
\para{ \f{\a}{n} = \f{\b}{n} \,\,\, \Rightarrow \,\,\, F(\a) = F(\b)};
\]
\item\textit{bounded} if
\[
\exists N \, 
\forall \a \para{ F(\a) \le N};
\]
\item \textit{constant} if
\[
\exists n \,  \para{ F \equiv n}.
\]
\end{itemize}

\begin{lemma} If $F: \X \to \N$ is uniformly continuous,
\[
F \equiv 0  \,\,\, \lor \,\,\, \lnot \para{F \equiv 0}.
\] 
\end{lemma}

\begin{lemma}\label{charactpcuc} Fix $F:\X \to \N$.
\begin{enumerate}[(a)]
\item $F$ is pointwise continuous  $\,\,\, \Leftrightarrow \,\,\,$ $\forall \a
  \, \exists n \para{F_{\f{\a}{n}} \text{ is constant}}$
\item $F$ is uniformly continuous $\,\,\, \Leftrightarrow \,\,\,$
$ \exists N \, \forall u \in \zo^N \para{F_u \text{ is constant}}$
\end{enumerate}
\end{lemma}

\begin{lemma}\label{bbaarr} Suppose that $F:\X \to \N$ is pointwise continuous. Then the set 
\[
\brac{u \mid F(u) \le\av{u}}
\]
is a bar.
\end{lemma}

\begin{proof} Fix $\a$. There exists $N$ such that
\[
\forall \b 
\para{ \f{\a}{N} = \f{\b}{N} \,\,\, \Rightarrow \,\,\, F(\a) = F(\b)}.
\]
We may assume that $F(\a) \le N$. Note that
\[
F(\f{\a}{N}) = F(\a) \le N.
\]
This implies that
$$\f{\a}{N} \in \brac{u \mid F(u) \le\av{u}}.$$

\end{proof}

\begin{lemma}\label{bounded} Every uniformly continuous function $F:\X \to \N$ is bounded.
\end{lemma}

We now consider the \textit{uniform continuity theorem}\index{uniform continuity theorem}:

\begin{description}
\item[$\uc$] Every pointwise continuous function $F:\X \to \N$ is uniformly continuous.
\end{description}

\begin{proposition}\label{ucfan} $\uc \,\,\,  \Rightarrow \,\,\, \fan$
\end{proposition}

\begin{proof} Let $B$ be a detachable bar.
The function
\begin{equation}\label{functionn}
F(\a) = \min\brac{n \mid  \f{\a}{n} \in B}
\end{equation}
is pointwise continuous.
In view of $\uc$, the function $F$ is uniformly continuous and therefore, by Lemma
\ref{bounded}, bounded. Thus, $B$ is a uniform bar.
\end{proof}

Fix functions $F,M:\X \to \N$. 
The function $M$ is a  \textit{modulus}\index{modulus of continuity} for $F$ if
\begin{enumerate}[(a)]
\item $M$ is pointwise continuous;
\item 
$\forall \a, \b,  \para{ \f{\a}{M(\a)} = \f{\b}{M(\a)} \,\,\, \Rightarrow \,\,\, F(\a) = F(\b)}$.
\end{enumerate}

\begin{description}
\item[$\muc$] Fix $F,M:\X \to \N$. Assume that $M$ is a pointwise continuous modulus of
$F$.
  Then $F$ is uniformly continuous.
\end{description}

\begin{proposition}\label{stier} $\muc \,\,\, \Leftrightarrow \,\,\, \fan$
\end{proposition}

\begin{proof} ``$\Rightarrow$'' Let $B$ be a detachable bar.
Then the function $F$, given by \eqref{functionn},
is a pointwise continuous modulus of itself, which implies that 
$F$ is uniformly continuous and therefore bounded.
Thus, $B$ is a uniform bar.

\bigskip

``$\Leftarrow$'' Fix $F,M:\X \to \N$ and assume that $M$ is a pointwise continuous modulus of 
$F$.
By Lemma \ref{bbaarr}, the set 
\[
B = \brac{u \mid M(u) \le\av{u}}
\]
is a bar and therefore, in view of $\fan$, a uniform bar.
Thus, there exists $N$ such that
\[
\forall \a \, \exists n \le N \para{ M(\f{\a}{n}) \le n }.
\]
This implies
\[
\forall \a \para{ F_{\f{\a}{N}} \text{ is constant}}.
\]
By Lemma \ref{charactpcuc}, we can conclude that $F$ is uniformly continuous.
\end{proof}

The axiom $\muc$ and Proposition \ref{stier} are taken from
\cite{the_fan_pale}.
For more about uniform continuity and $\fan$ see
\cite{dbfr}.

\section{The fan theorem for c-sets}

In the previous section, we have matched
uniform continuity with the fan theorem by 
replacing $\uc$ with its weaker version $\muc$.
Another option is to strengthen
$\fan$ by relaxing the requirement that the bar be detachable.
This leads to the axiom $\cfan$, which was introduced in \cite{the_logical_pale}.

\begin{lemma}\label{charinnen} For all subsets $A$ of $\Xs$ an all $u$ we have
\begin{enumerate}[(a)]
\item $u \in A^\circ \Leftrightarrow A_u = \Xs$ and
\item $(A_u)^\circ=(A^\circ)_u$.
\end{enumerate}
\end{lemma}

\begin{proof} In order to prove (b), fix $u$, $v$ and $w$ and note that
\[
w \in (A_u)^\circ \Leftrightarrow 
\forall v \para{w * v \in A_u} \Leftrightarrow 
\]
\[
\forall v \para{u * (w * v) \in A} \Leftrightarrow 
\forall v \para{(u * w) * v \in A} \Leftrightarrow 
\]
\[
u * w \in A^\circ \Leftrightarrow 
w \in (A^\circ)_u.
\]
\end{proof}

\begin{description}
\item[$\cfan$] Every\index{fan theorem for c-sets} bar which is a c-set is a uniform bar.
\end{description}

We introduce the auxiliary axioms $\mr{DEFU}$ (\textbf{de}cidability whether certain bars are \textbf{fu}ll)
and $\mr{DECO}$ (\textbf{de}cidability whether certain functions are \textbf{co}nstant).

\begin{description}
\item[$\mr{DEFU}$] Let $D$ be a detachable subset of 
$\Xs$ such that $D^\circ$ is a bar. Then,
$$\exists u \para{u \notin D}  \lor \lnot \exists u \para{u \notin D}.$$
\item[$\mr{DECO}$] Let $F:\X \to \N$ be pointwise continuous. Then,
 $$\exists \a, \b \para{F(\a) \ne F(\b)}  \lor
 \lnot \exists \a, \b \para{F(\a) \ne F(\b)}. $$
\end{description}

\begin{lemma}\label{FOCO}
$\mr{DEFU}  \Leftrightarrow  \mr{DECO}$
\end{lemma}

\begin{proof} ``$\Rightarrow$'' Fix a pointwise continuous function $F : \X \to \N$.
Set 
\[
D = \brac{ u \mid F(u)=F(u*(1))}.
\]
By Lemma \ref{charactpcuc}, $D^\circ$ is a bar.
Note that
\[
\exists u \para{u \notin D} \Leftrightarrow \exists  \a,\b \para{F(\a) \ne F(\b)}.
\]

``$\Leftarrow$'' Fix $D \subseteq \Xs$ and suppose that $D^\circ$ is a bar. Define
\[
F(\a) = \min \brac{n \mid \forall m \ge n \para{\f{\a}{m} \in D}}
\]
and note that
\begin{equation}\label{dumm}
\exists u \para{u \notin D} \Leftrightarrow \exists  \a \para{F(\a) > 0 )}.
\end{equation}
In presence of $\mr{DECO}$, the right hand side of \eqref{dumm} is decidable, which implies that
the left hand side of \eqref{dumm} is decidable.
\end{proof}

The diagram in the next proposition sheds some light on the relationship
between bars and functions. Its core part is the equivalence
between $\cfan$ and $\uc$. 

\begin{proposition}\label{pic2}
\[
\begin{array}{ccccccc}
  \wkl & \Rightarrow &\fan \land\mr{DEFU} & \Leftrightarrow &\cfan&\Rightarrow&\fan \\
  &&&&&&\\
  
  & & \Updownarrow &&  \Updownarrow && \\
    &&&&&&\\
& &\fan \land \mr{DECO} & \Leftrightarrow&\uc &&
\end{array}
\]
\end{proposition}

\begin{proof} ``$\wkl \Rightarrow \mr{DEFU}$'' Let $D$ be a detachable subset $D$ of $\Xs$ such that $D^\circ$ is a bar.
Define an infinite tree $T$ by
\[
u \in T \Leftrightarrow 
 \forall w \para{%
   \av{w} \le \av{u} \land w \notin D \Rightarrow \exists k \le \av{w} \para{\f{u}{k} \notin D} }.
\]
For each $u \in T$ of length $n$ we have
\[
\brac{\f{u}{k} \mid k \le n }  \subseteq D \Leftrightarrow \brac{w \mid \av{w} \le n } \subseteq D.
\]
Let $\a$ be a path of $T$.
The following are equivalent:
\begin{enumerate}[(i)]
\item $\exists u \para{u \notin D}$
\item $\exists n \para{\f{\a}{n} \notin D}$  
\end{enumerate}
There exists $N$ such that $\f{\a}{N} \in D^\circ$.
Thus, $(ii)$ is decidable, which yields the decidability of $(i)$.

\bigskip

``$\fan \land\mr{DEFU}  \Rightarrow \cfan$'' Fix a detachable subset $D$ of 
$\Xs$ such that $D^\circ$ is a bar.
We show that  $D^\circ$ is detachable.
To this end, fix some $u$. Note that
$(D^\circ)_u$ is a bar and therefore
$(D_u)^\circ$ is a bar.
We can conclude that
\[
D_u = \Xs \lor \exists w \para{w \notin D_u},
\]
which implies
\[
u \in D^\circ  \lor u \notin D^\circ.
\]

\bigskip

``$\fan \land \mr{DECO}  \Rightarrow \uc$''
Let $F: \X \to \N$ be pointwise continuous.
Set
\[
B = \brac{u \mid F_u \text{ is constant}}.
\]
In view of $\mr{DECO}$, the set $B$ is detachable.
Since $F$ is pointwise continuous, $B$ is a bar.
Thus, $B$ is a uniform bar, which implies the
uniform continuity of $F$.

\bigskip

``$\uc \Rightarrow \fan \land \mr{DECO}$'' The implication
$\uc \Rightarrow \fan$ 
holds by Proposition \ref{ucfan}. The implication $\uc \Rightarrow \mr{DECO}$ follows
from Lemma $\ref{charactpcuc}$.

\bigskip

``$\cfan  \Rightarrow \fan$'' Let $B$ be a detachable bar which is closed under extension.
Then $B = B^\circ$, which implies that $B$ is a c-set and therefore, in presence of $\cfan$,  a uniform bar.

\end{proof}

Another equivalent of 
$\cfan$ is the so-called anti-Specker
theorem;
see \cite{berger_bridges_a_fan_pale}.
For more on the fan theorem, see \cite{ludin} and \cite{wim}.

\bibliographystyle{plain}
\bibliography{fan.bib} 


\end{document}